\documentclass[reqno]{amsart} 

\usepackage[latin1]{inputenc}

\usepackage{enumitem}

\usepackage{amsmath,amssymb,bbm}
\usepackage{color,epsfig,psfrag,url}

\usepackage[colorlinks,bookmarks]{hyperref}

\let\version\today

\let\subset\subseteq
\let\eps\varepsilon

\def\rom{\rm (\roman{*})}
\def\abc{\rm (\alph{*})}
\def\arab{\rm (\arabic{*})}

\newtheorem{theorem}                   {Theorem}
\newtheorem{lemma}           [theorem] {Lemma}   
\newtheorem{corollary}       [theorem] {Corollary}   
\newtheorem{proposition}     [theorem] {Proposition}  
\newtheorem{claim}           [theorem] {Claim}

\newtheorem{definition}      [theorem] {Definition} 
 
 \theoremstyle{remark}

\newlength{\itemleftmargin}
\newcommand{\makeversion}{
  \let\oldfootnote\thefootnote %
  \renewcommand{\thefootnote}{} %
  \footnote{\version} %
  \let\thefootnote\oldfootnote %
}

\newcommand{\comment}[1]{}

\newcommand{\field}[1]{\mathbb{#1}}
\newcommand{\N}{\field{N}}

\newcommand{\cC}{\mbox{${ \mathcal C }$}}

\newcommand{\cG}{\mbox{${ \mathcal G }$}}

\newcommand{\Then}{\Rightarrow}
\newcommand{\bigO}{\mathcal{O}}

\newcommand{\deff}{\mathrel{\mathop:}=}

\newcommand{\bdd}[1][\eps]{\operatorname{b}_{#1}}

\DeclareMathOperator{\trw}{tw} 
\DeclareMathOperator{\bdw}{bw} 
\DeclareMathOperator{\s}{s} 
\DeclareMathOperator{\dist}{dist}

\DeclareMathOperator{\dcup}{\dot\cup} 

 \title[Bandwidth, expansion, treewidth, separators, and universality]{
	Bandwidth, expansion, treewidth, separators, and universality for bounded
	degree graphs}

  \author{Julia B\"ottcher} 
  \address{Zentrum Mathematik, Technische Universit\"at M\"unchen, 
    Boltzmannstra\ss{}e~3, D-85747 Garching bei M\"unchen, Germany} 
  \email{boettche@ma.tum.de}

  \author{Klaas P. Pruessmann}
  \address{Institute for Biomedical Engineering, 
    University and ETH Zurich, 
    Gloriastr. 35, 8092, Z\"urich, Switzerland}
  \email{pruessmann@biomed.ee.ethz.ch}

  \author{Anusch Taraz } 
  \address{Zentrum Mathematik, Technische Universit\"at M\"unchen, 
    Boltzmannstra\ss{}e~3, D-85747 Garching bei M\"unchen, Germany} 
  \email{taraz@ma.tum.de}

  \author{Andreas W\"urfl }
  \address{Zentrum Mathematik, Technische Universit\"at M\"unchen, 
    Boltzmannstra\ss{}e~3, D-85747 Garching bei M\"unchen, Germany} 
  \email{wuerfl@ma.tum.de}

\thanks{ The first and third author were partially supported by DFG grant TA
309/2-1. The first author was partially supported by a Minerva grant.}

\date{\version}

\begin{document}


\begin{abstract}
  We establish relations between the bandwidth and the treewidth of bounded  
  degree graphs $G$, and relate these parameters to   the size of a separator
  of $G$ as well as the size of an expanding subgraph of $G$. Our results imply
  that if one of these parameters is sublinear in the number of vertices of $G$
  then so are all the others. This implies for example that graphs of fixed
  genus have sublinear bandwidth or, more generally, a corresponding result for
  graphs with any fixed forbidden minor. As a consequence we establish a simple
  criterion for universality for such classes of graphs and show for example
  that for each $\gamma>0$ every $n$-vertex graph with minimum degree
  $(\frac34+\gamma)n$ contains a copy of every bounded-degree planar graph on
  $n$ vertices if $n$ is sufficiently large.
\end{abstract}

\maketitle
\let\languagename\relax   


\section{Introduction}

There are a number of different parameters in graph theory which measure how well
a graph can be organised in a particular way, where the type of desired
arrangement is often motivated by geometrical properties, algorithmic
considerations, or specific applications. Well-known examples of such parameters
are the genus, the bandwidth, or the treewidth of a graph. The topic of this paper is
to discuss the relations between such parameters. We would like to determine how
they influence each other and what causes them to be large. To this end we will
mostly be interested in distinguishing between the cases when these parameters
are linear or sublinear in the number of vertices in the graph.

\medskip

We start with a few simple observations. Let $G=(V,E)$ be a graph on $n$
vertices. The \emph{bandwidth} of $G$ is denoted by $\bdw(G)$ and defined to be
the minimum positive integer $b$, such that there exists a labelling of the
vertices in $V$ by numbers $1,\dots,n$ so that the labels of every pair of
adjacent vertices differ by at most $b$. Clearly one reason for a graph to have
high bandwidth are vertices of high degree as $\bdw(G)\ge
\lceil\Delta(G)/2\rceil$, where $\Delta(G)$ is the maximum degree of $G$. It is
also clear that not all graphs of bounded maximum degree have sublinear
bandwidth: Consider for example a random bipartite graph $G$ on $n$ vertices with
bounded maximum degree. Indeed, with high probability, $G$ does not have small
bandwidth since in any linear ordering of its vertices there will be an edge
between the first $n/3$ and the last $n/3$ vertices in this ordering. The reason
behind this obstacle is that $G$ has good expansion properties (definitions and
exact statements are provided in Section~\ref{sec:results}). This implies that
graphs with sublinear bandwidth
cannot exhibit good expansion properties. One may ask whether the converse is
also true, i.e., whether the absence of big expanding subgraphs in bounded-degree
graphs must lead to small bandwidth. We will prove that this is indeed the case
via the existence of certain separators.

In fact, we will show a more general theorem in Section~\ref{sec:results}
(Theorem~\ref{thm:main}) which proves that the concepts of sublinear bandwidth,
sublinear treewidth, bad expansion properties, and sublinear separators are
equivalent for graphs of bounded maximum degree. In order to prove this
theorem, we will establish quantitative relations between the parameters involved
(see Theorem~\ref{thm:sep-band}, Theorem~\ref{thm:bound-trw}, and
Proposition~\ref{prop:band-bound}).

\medskip

As a byproduct of these relations we obtain sublinear bandwidth bounds for
several graph classes (see Section~\ref{sec:appl}). Since planar graphs are known
to have small separators~\cite{LipTar} for example, we get the following
result: The bandwidth of a planar graph on~$n$ vertices with maximum degree at
most~$\Delta$ is bounded from above by $\bdw(G) \le \frac{15n}{\log_\Delta(n)}$.
This extends a result of Chung~\cite{Chu} who proved that any $n$-vertex tree $T$
with maximum degree $\Delta$ has bandwidth at most $5n/\log_\Delta(n)$. Similar
upper bounds can be formulated for graphs of any fixed genus and, more generally,
for any graph class defined by a set of forbidden minors (see
Section~\ref{subsec:separators}).

In Section~\ref{subsec:universal} we conclude by considering
applications of these results to the domain of universal graphs and derive
implications such as the following. If $n$ is sufficiently large then any
$n$-vertex graph with minimum degree slightly above $\frac34n$ contains every
planar $n$-vertex graphs with bounded maximum degree as a subgraph (cf.\
Corollary~\ref{cor:planar-univers}).

\section{Definitions and Results}
\label{sec:results}


In this section we formulate our main results which provide relations between
the bandwidth, the treewidth, the expansion properties, and separators
of bounded degree graphs.
We need some further definitions.
For a graph $G=(V,E)$ and disjoint vertex sets $A,B\subset V$ we
denote by $E(A,B)$ the set of edges with one vertex in $A$ and one vertex in
$B$ and by $e(A,B)$ the number of such edges. 

Next, we will introduce the notions of \emph{tree
decomposition} and \emph{treewidth}. Roughly speaking, a tree decomposition
tries to arrange the vertices of a graph in a tree-like manner and the
treewidth measures how well this can be done.

\begin{definition}[treewidth]
  A tree decomposition of a graph $G=(V,E)$ is a pair $\left(\{X_i: i\in
    I\},\right.$ $\left.T=(I,F)\right)$ where $\{X_i: i \in I\}$ is a family of
    subsets $X_i\subseteq V$ and $T = (I,F)$ is a tree such that the following holds:
  \begin{enumerate}[label=\abc,leftmargin=*,itemsep=0mm,parsep=0mm,topsep=2mm]
    \item $\bigcup_{i \in I} X_i = V$,
    \item for every edge $\{v,w\} \in E$ there exists $i \in I$ with $\{v,w\}
      \subseteq X_i$,
    \item for every $i,j,k \in I$ the following holds: if $j$ lies on the path
    from $i$ to $k$ in $T$, then $X_i \cap X_k \subseteq X_j$.
  \end{enumerate}
  The width of $\left(\{X_i:i \in I\},T=(I,F)\right)$ is defined as
  $\max_{i \in I} |X_i| -1$. The tree\-width $\trw(G)$ of $G$ is
  the minimum width of a tree decomposition of $G$.
\end{definition}

It follows directly from the definition that $\trw(G)\le \bdw(G)$ for any
graph~$G$: if the vertices of $G$ are labelled by numbers $1,\dots,n$ such that
the labels of adjacent vertices differ by at most $b$, then $I:=[n-b]$,
$X_i:=\{i,\dots,i+b\}$ for $i\in I$ and $T:=(I,F)$ with $F:=\{\{i-1,i\}:2\le i
\le n-b\}$ define a tree decomposition of $G$ with width $b$.

A \emph{separator} in a graph is a small cut-set that splits the graph into
components of limited size.

\begin{definition}[separator, separation number] 
  Let $\tfrac12 \le \alpha < 1$ be a real number, $s\in\N$ and $G=(V,E)$ a
  graph. 
  A subset $S\subset V$ is said to be an $(s,\alpha)$-separator of $G$,
  if there exist subsets $A,B \subset V$ such that
  \begin{enumerate}[label=\abc,leftmargin=*,itemsep=0mm,parsep=0mm,topsep=2mm]
    \item $V = A \dcup B \dcup S$,
    \item $|S| \leq s$, $|A|, |B| \leq \alpha |V|$, and
    \item $E(A,B) = \emptyset$.
  \end{enumerate}
  We also say that $S$ separates $G$ into $A$ and $B$.
  The separation number $\s(G)$ of $G$ is the smallest $s$ such that all
  subgraphs $G'$ of $G$ have an $(s,2/3)$-separator.
\end{definition}

A vertex set is said to be expanding, if it has many external neighbours.
We call a graph \emph{bounded}, if every sufficiently large subgraph
contains a subset which is not expanding.

\begin{definition}[expander, bounded]
  Let $\eps>0$ be a real number, $b\in\N$ and 
consider graphs $G=(V,E)$ and $G'=(V',E')$. 
We say that $G'$ is an $\eps$-expander if all subsets $U\subset
  V'$ with $|U|\le |V'|/2$ fulfil $|N(U)|\ge \eps |U|$.  (Here $N(U)$ is
  the set of neighbours of vertices in $U$ that lie outside of $U$.)
  The graph $G$ is called $(b,\eps)$-bounded, if no subgraph $G'\subset
  G$ with $|V'| \ge b$ vertices is an $\eps$-expander.
  Finally, we define the $\eps$-boundedness $\bdd(G)$ of $G$ to be the minimum
  $b$ for which $G$ is $(b+1,\eps)$-bounded.
\end{definition}

There is a wealth of literature on expander graphs (see,
e.g., \cite{HooLinWig}). In particular, it is known that for example
(bipartite) random graphs with bounded maximum degree form a family of
$\eps$-expanders. We also loosely say that such graphs
have~\emph{good expansion properties}.
%
%

\medskip

As indicated earlier, our aim is to provide relations between the
parameters we defined above. A well known example of a result of this type is
the following theorem due to Robertson and Seymour which relates the
treewidth and the separation number of a graph.\footnote{In fact, their result
states that any graph $G$ has a $(\trw(G)+1,1/2)$-separator, and does not talk
about subgraphs of $G$. But since every
subgraph of $G$ has treewidth at most $\trw(G)$, it thus also has a
$(\trw(G)+1,1/2)$-separator and the result, as stated here, follows.}

\begin{theorem}[treewidth$\to$separator, \cite{RobSey_minors5}]    
\label{thm:trw-sep}
All graphs $G$ have separation number
\begin{equation*}
\s(G) \le \trw(G)+1.
\end{equation*}
\end{theorem} 

This theorem states that graphs with small treewidth have small separators.
By repeatedly extracting separators, one can show that (a
qualitatively different version of) the converse also holds:
$\trw(G)\le\bigO(\s(G)\log n)$ for a graph $G$ on $n$ vertices (see,
e.g., \cite[Theorem 20]{Bod}). In this paper, we use a similar but more
involved argument to show that one can establish the following relation linking the
separation number with the bandwidth of graphs with bounded maximum degree.

\begin{theorem}[separator$\to$bandwidth]   
\label{thm:sep-band}
 For each~$\Delta\ge 2$ every graph $G$ on $n$ vertices with maximum degree
 $\Delta(G) \le \Delta$ has bandwidth
\begin{equation*}
  \bdw(G) \le \frac{6n}{\log_{\Delta} (n/\s(G))}.  
\end{equation*}
\end{theorem}

The proof of this theorem is provided in Section~\ref{sec:sep-band}. 
Observe that Theorems~\ref{thm:trw-sep} and~\ref{thm:sep-band} together with the 
obvious inequality $\trw(G)\le \bdw(G)$ tie the concepts of treewidth, bandwidth,
and separation number well together. 
Apart from the somewhat negative statement of \emph{not having} a small separator,
what can we say about a graph with large tree- or bandwidth? 
The next theorem states that such a graph must contain a big expander.


\begin{theorem}[bounded$\to$treewidth]     
\label{thm:bound-trw}
  Let $\eps > 0$ be constant. All graphs $G$ on $n$ vertices have treewidth
  $\trw(G) \le 2\bdd(G) + 2\eps n$.
%
\end{theorem}

A result with similar implications was recently proved by Grohe and Marx
in~\cite{GroMar}. It shows that $b_{\eps}(G) < \eps n$ implies $\trw(G) \le 2\eps
n$.
For the sake of being self-contained we present our (short) proof of
Theorem~\ref{thm:bound-trw} in Section~\ref{sec:bound}. In addition, it is not
difficult to see that conversely the boundedness of a graph can be estimated via
its bandwidth -- which we prove in Section~\ref{sec:bound}, too.

\begin{proposition}[bandwidth$\to$bounded]     
\label{prop:band-bound}
 Let $\eps>0$ be constant. All graphs $G$ on $n$ vertices have $\bdd(G)\le 2
 \bdw(G)/\eps$.
\end{proposition}

A qualitative consequence summarising the four results above is given in the following theorem. It states that if one
of the parameters bandwidth, treewidth, separation number, or boundedness is
sublinear for a family of bounded degree graphs, then so are the others. 

\begin{theorem}[sublinear equivalence theorem]
  \label{thm:main}
  Let $\Delta$ be an arbitrary but fixed positive integer and consider a
  hereditary class of graphs $\cC$ such that all graphs in $\cC$ have maximum
  degree at most $\Delta$. Denote by $\cC_n$ the set of those graphs in $\cC$
  with $n$ vertices. Then the following four properties are equivalent:
  \begin{enumerate}[label=\arab,leftmargin=*,itemsep=0mm,parsep=0mm,topsep=2mm]
  \item\label{it:main:1}
    For all $\beta_1>0$ there is $n_1$ s.t. $\trw(G)\le\beta_1 n$ for all
    $G\in\cC_n$ with $n\ge n_1$.
  \item\label{it:main:2}
    For all $\beta_2>0$ there is $n_2$ s.t. $\bdw(G)\le\beta_2 n$ for all
    $G\in\cC_n$ with $n\ge n_2$.
  \item\label{it:main:3}
    For all $\beta_3,\eps>0$ there is $n_3$ s.t. $\bdd(G)\le\beta_3 n$ for all
    $G\in\cC_n$ with $n\ge n_3$.
  \item\label{it:main:4}
    For all $\beta_4>0$ there is $n_4$ s.t. $\s(G)\le\beta_4 n$ for all
    $G\in\cC_n$ with $n\ge n_4$.
  \end{enumerate}
\end{theorem}

The paper is organized as follows. Section~\ref{sec:proofs} contains the proofs
of all the results mentioned so far: First we derive Theorem~\ref{thm:main}
from Theorems~\ref{thm:trw-sep}, \ref{thm:sep-band},
\ref{thm:bound-trw} and Proposition~\ref{prop:band-bound}. Then
Section~\ref{sec:sep-band} is devoted to the proof of
Theorem~\ref{thm:sep-band}, whereas Section~\ref{sec:bound} contains the proofs
of Theorem~\ref{thm:bound-trw} and Proposition~\ref{prop:band-bound}.
Finally, in Section~\ref{sec:appl}, we apply our results to deduce that certain
classes of graphs have sublinear bandwidth and can therefore be embedded as
spanning subgraphs into graphs of high minimum degree.

\section{Proofs}\label{sec:proofs}

\subsection{Proof of Theorem~\ref{thm:main}}

\begin{proof} 
  \ref{it:main:1}$\Then$\ref{it:main:4}:\quad
  Given $\beta_4 > 0$ set $\beta_1\deff\beta_4/2$, let $n_1$ be the
  constant from~\ref{it:main:1} for this $\beta_1$, and set $n_4 \deff 
  \max\{n_1,2/\beta_4\}$. Now consider $G\in\cC_n$ with $n \ge n_4$. By
  assumption we have $\trw(G)\le\beta_1 n$ and thus we can apply
  Theorem~\ref{thm:trw-sep} to conclude that
  $\s(G)\le\trw(G) + 1\le\beta_1 n + 1\le(\beta_4/2+1/n)n
   \le \beta_4 n$.
   
  \ref{it:main:4}$\Then$\ref{it:main:2}:\quad
  Given $\beta_2 > 0$ let $d:=\max\{2,\Delta\}$, set $\beta_4\deff
  d^{-6/\beta_2}$, get $n_4$ from~\ref{it:main:4} for this~$\beta_4$, and set
  $n_2\deff n_4$. Let $G\in\cC_n$ with $n \ge n_2$. We conclude
  from~\ref{it:main:4} and Theorem~\ref{thm:sep-band} that
  \begin{equation*}
    \bdw(G)\le\frac{6n}{\log_{d}n-\log_{d}\s(G)}
    \le\frac{6n}{\log_{d}n-\log_{d}(d^{-6/\beta_2}n)} 
    = \beta_2 n.
  \end{equation*}

  \ref{it:main:2}$\Then$\ref{it:main:3}:\quad
  Given $\beta_3, \eps > 0$ set $\beta_2\deff\eps\beta_3/2$, get
  $n_2$ from~\ref{it:main:2} for this $\beta_2$ and set $n_3\deff n_2$.
  By~\ref{it:main:2} and Proposition~\ref{prop:band-bound} we get for
  $G\in\cC_n$ with $n \ge n_3$ that
  $\bdd(G)\le2\bdw(G)/\eps\le2\beta_2n/\eps\le\beta_3 n$.
  
  \ref{it:main:3}$\Then$\ref{it:main:1}:\quad
  Given $\beta_1>0$, set $\beta_3\deff\beta_1/4$, $\eps\deff\beta_1/4$ and
  get $n_3$ from~\ref{it:main:3} for this $\beta_3$ and $\eps$, and set
  $n_1 \deff n_3$. Let $G\in\cC_n$ with $n \ge n_1$. 
  Then~\ref{it:main:3} and Theorem~\ref{thm:bound-trw} imply
  $\trw(G)\le2\bdd(G)+2\eps n\le2\beta_3n+2(\beta_1/4)n=\beta_1n$.
\end{proof}

\subsection{Separation and bandwidth}\label{sec:sep-band}

For the proof of Theorem~\ref{thm:sep-band} we will use the following
decomposition result which roughly states the following. If the removal of
a small separator $S$ decomposes the vertex set of a graph $G$ into relatively
small components $R_i\dcup P_i$ such that the vertices in $P_i$ form a ``buffer''
between the vertices in the separator $S$ and the set of remaining vertices
$R_i$ in the sense that $\dist_G(S,R_i)$ is sufficiently big,
then the bandwidth of $G$ is small. 

\begin{lemma}[decomposition lemma]
  \label{lem:decomposition}
	Let $G=(V,E)$ be a graph and $S$, $P$, and $R$ be vertex sets
	such that $V=S\dcup P\dcup R$. For $b,r\in\N$ with $b\ge 3$ assume further that there are
	decompositions $P=P_1\dcup\dots\dcup P_b$ and
	$R=R_1\dcup\dots\dcup R_b$ of $P$ and $R$, respectively, such that the
	following properties are satisfied:
	\begin{enumerate}[label=\rom,leftmargin=*,itemsep=0mm,parsep=0mm,topsep=2mm]
      \item\label{it:dec:i} $|R_i|\le r$,
      \item\label{it:dec:ii} $e(R_i\dcup P_i,R_j\dcup P_j)=0$ for all $1\le
      i<j\le b$,
      \item\label{it:dec:iii} $\dist_G(u,v)\ge \lfloor b/2\rfloor$ for all $u\in
      S$ and $v\in R_i$ with $i\in[b]$.
    \end{enumerate}
	Then $\bdw(G)< 2(|S|+|P|+r)$.
\end{lemma}
\begin{proof}
Assume we have $G=(V,E)$, $V=S\dcup P\dcup R$ and $b,r \in \N$ with the
properties stated above. 
Our first goal is to partition $V$ into pairwise disjoint sets $B_1,\dots,B_b$,
which we call \emph{buckets}, and that satisfy the following property:
\begin{equation}\label{eq:dec:edges}
  \text{If $\{u,v\}\in E$ for $u\in B_i$ and $v\in B_j$ then $|i-j|\le 1$.}
\end{equation}
To this end all vertices of $R_i$ are placed into bucket $B_i$ for each
$i \in [b]$ and the vertices of $S$ are placed into bucket $B_{\lceil
b/2\rceil}$. The remaining vertices from the sets $P_i$ are distributed over
the buckets according to their distance from $S$: vertex $v \in P_{i}$ is
assigned to bucket $B_{j(v)}$ where $j(v)\in [b]$ is defined by
\begin{equation}\label{eq:dec:jv}
  j(v):=\begin{cases}
    i & \text{ if } \dist(S,v) \ge |\lceil b/2\rceil-i|, \\    
    \lceil b/2\rceil - \dist(S,v) & \text{ if } \dist(S,v) < \lceil b/2\rceil-i,
    \\ \lceil b/2\rceil + \dist(S,v) & \text{ if } \dist(S,v) <
    i-\lceil b/2\rceil.
  \end{cases}
\end{equation}
This placement obviously satisfies 
\begin{equation}\label{eq:dec:B}
  |B_i|\le|S|+|P|+|R_i|\le|S|+|P|+r
\end{equation}
by construction and condition~\ref{it:dec:i}.
Moreover, we claim that it guarantees condition~\eqref{eq:dec:edges}. Indeed, let $\{u,v\}\in E$
be an edge. If $u$ and $v$ are both in $S$ then clearly~\eqref{eq:dec:edges} is
satisfied. Thus it remains to consider the case where, without loss of
generality, $u\in R_i\dcup P_i$ for some $i\in [b]$. By condition~\ref{it:dec:ii} this
implies $v\in S\dcup R_i\dcup P_i$. First assume that $v\in S$. Thus
$\dist(u,S)=1$ and from condition~\ref{it:dec:iii} we infer that $u\in P_i$.
Accordingly $u$ is placed into bucket
$B_{j(u)}\in\{B_{\lceil b/2\rceil-1},B_{\lceil b/2\rceil},B_{\lceil
b/2\rceil+1}\}$ by~\eqref{eq:dec:jv} and $v$ is placed into
bucket~$B_{\lceil b/2\rceil}$ and so we also get~\eqref{eq:dec:edges} in this
case. If both $u,v\in R_i\dcup P_i$, on the other hand, we are clearly done if
$u,v\in R_i$. So assume without loss of generality, that $u\in P_i$. If $v\in
P_i$ then we conclude from $|\dist(S,u)-\dist(S,v)|\le 1$
and~\eqref{eq:dec:jv} that $u$ is placed into bucket $B_{j(u)}$ and $v$ into
$B_{j(v)}$ with $|j(u)-j(v)|\le 1$. If $v\in R_i$, finally, observe
that $|\dist(S,u)-\dist(S,v)|\le 1$ together with condition~\ref{it:dec:iii}
implies that $\dist(S,u)\ge \lfloor b/2\rfloor-1$ and so $u$ is placed into
bucket $B_{j(u)}$ with $|j(u)-i|\le 1$, where $i$ is the index such that $v\in B_i$,
by~\eqref{eq:dec:jv}.
Thus we also get~\eqref{eq:dec:edges} in this last case.

Now we are ready to construct an ordering of~$V$ 
respecting the desired bandwidth bound. We start with the vertices in bucket
$B_1$, order them arbitrarily, proceed to the vertices in bucket $B_2$, order
them arbitrarily, and so on, up to bucket $B_b$. By
condition~\eqref{eq:dec:edges} this gives an ordering with bandwidth at most
twice as large as the largest bucket and thus we conclude from~\eqref{eq:dec:B}
that $\bdw(G) < 2(|S|+|P|+r)$.
\end{proof}

A decomposition of the vertices of $G$ into buckets as in the proof of
Lemma~\ref{lem:decomposition} is also called a path partition of $G$ and
appears, e.g., in~\cite{DujSudWoo}.

Before we get to the proof of Theorem~\ref{thm:sep-band}, we will establish 
the following technical observation about labelled trees.

\begin{proposition}\label{prop:tree}
  Let~$b$ be a positive real,
  $T=(V,E)$ be a tree with $|V|\ge 3$, and $\ell:V\to[0,1]$ be a real valued labelling of its
  vertices such that $\sum_{v\in V}\ell(v)\le 1$.
  Denote further for all $v\in V$ by $L(v)$ the set of leaves that 
are adjacent to $v$ and suppose that $\ell(v)+\sum_{u\in L(v)}\ell(u)\ge|L(v)|/b$.
  Then $T$ has at most $b$ leaves in total.
\end{proposition}
\begin{proof}
  Let $L\subset V$ be the set of leaves of $T$  and $I:=V\setminus L$ be the set
  of internal vertices. Clearly
  \begin{equation*}
    1 \ge \sum_{v\in V}\ell(v)
    =\sum_{v\in I}\left(\ell(v)+\sum_{u\in L(v)}\ell(u)\right)
    \ge\sum_{v\in I}\frac{|L(v)|}{b}
    =\frac{|L|}{b}
  \end{equation*}
  which implies the assertion.
\end{proof}

The idea of the proof of Theorem~\ref{thm:sep-band} is to repeatedly extract
separators from~$G$ and the pieces that result from the removal of such
separators. We denote the union of these separators by $S$, put all remaining
vertices with small distance from $S$ into sets $P_i$, and all other
vertices into sets $R_i$. Then we can apply the decomposition lemma
(Lemma~\ref{lem:decomposition}) to these sets $S$, $P_i$, and $R_i$. This,
together with some technical calculations, will give the desired bandwidth bound
for~$G$.

\begin{proof}[of Theorem~\ref{thm:sep-band}] Let $G = (V,E)$ be a graph on $n$
vertices with maximum degree $\Delta\ge 2$.
Observe that the desired bandwidth bound is trivial if $\Delta=2$ or if
$\log_{\Delta} n -\log_{\Delta}\s(G)\le 6$, so assume in the following that 
$\Delta \ge 3$ and $\log_{\Delta} n -\log_{\Delta}\s(G)>6$. Define
\begin{equation}\label{eq:sep-band:bt}
  \beta:=\log_{\Delta} n -\log_{\Delta} \s(G)
  \qquad\text{and}\qquad
  b:=\left\lfloor\beta\right\rfloor \ge 6
\end{equation}
and observe that with this choice of $\beta$ our aim is to show that $\bdw(G)\le
6n/\beta$.

The goal is to construct a partition $V = S\dcup P \dcup R$ with the
properties required by Lemma~\ref{lem:decomposition}. For this
purpose we will recursively use the fact that $G$ and its subgraphs have separators of
size at most $\s(G)$. 
In the $i$-th round we will identify separators $S_{i,k}$ in $G$ whose removal
splits $G$ into parts $V_{i,1},\dots,V_{i,b_i}$. 
%
The details are as follows.

In the first round let $S_{1,1}$ be an arbitrary $(\s(G),2/3)$-separator in $G$
that separates $G$ into $V_{1,1}$ and $V_{1,2}$ and set 
$b_1:=2$.
In the $i$-th round, $i>1$, consider each of the sets $V_{i-1,j}$ with
$j\in[b_{i-1}]$. If $|V_{i-1,j}|\le 2n/b$ then let
$V_{i,j'}:=V_{i-1,j}$, otherwise choose an $(\s(G),2/3)$-separator $S_{i,k}$ 
that separates $G[V_{i-1,j}]$ into sets $V_{i,j'}$ and $V_{i,j'+1}$ (where $k$
and $j'$ are appropriate indices, for simplicity we do not specify
them further). Let $S_i$ denote the union of all separators
constructed in this way (and in this round).
This finishes the $i$-th round. We stop this procedure as soon as all sets
$V_{i,j'}$ have size at most $2n/b$ and denote the corresponding $i$ by
$i^*$. Then $b_{i^*}$ is the number of sets $V_{i^*,j'}$ we end up with in the
last iteration. Let further $x_S$ be the number of separators $S_{i,k}$
extracted from~$G$ during this process in total.

\begin{claim}\label{cl:sep-band}
  We have $b_{i^*}\le b$ and $x_S\le b-1$.
\end{claim}

We will postpone the proof of this fact and first show how it implies the
theorem.
Set $S:=\bigcup_{i\in[i^*]}S_i$, for $j\in[b_{i^*}]$ define
\begin{equation*}
  P_j:=\{v \in V_{i^*,j} : \dist(v,S) < \lfloor b/2\rfloor\}
  \qquad\text{and}\qquad R_j = V_{i^*,j} \setminus P_j,
\end{equation*} 
set $P_j=R_j=\emptyset$  for $b_{i^*}<j\le b$ and finally define
$P:=\bigcup_{j\in[b]} P_j$ and $R:=\bigcup_{j\in[b]} R_j$.

We claim that $V=S\dcup P\dcup R$ is a partition that satisfies the requirements
of the decomposition lemma (Lemma~\ref{lem:decomposition}) with parameter~$b$
and~$r=2n/b$. To check this, observe first that for all $i\in[i^*]$ and
$j,j'\in[b_i]$ we have $e(V_{i,j},V_{i,j'})=0$ since $V_{i,j}$ and $V_{i,j'}$
were separated by some $S_{i',k}$. It follows that $e(R_j\dcup P_j,R_{j'}\dcup
P_{j'})=e(V_{i^*,j},V_{i^*,j'})=0$ for all $j,j'\in[b_{i^*}]$. Trivially
$e(R_j\dcup P_j,R_{j'}\dcup P_{j'})=0$ for all $j\in[b]$ and $b_{i^*}<j'\le b$
and therefore we get condition~\ref{it:dec:ii} of Lemma~\ref{lem:decomposition}.
Moreover, condition~\ref{it:dec:iii} is satisfied by the definition of the sets
$P_j$ and $R_j$ above. To verify condition~\ref{it:dec:i} note that
$|R_j|\le|V_{i^*,j}|\le2n/b=r$ for all $j\in[b_{i^*}]$ by the choice of~$i^*$ and
$|R_j|=0$ for all $b_{i^*}<j\le b$. Accordingly we can apply
Lemma~\ref{lem:decomposition} and infer that
\begin{equation}\label{eq:sep-band:bdw}
   \bdw(G)  \le 2\left(|S|+|P|+\frac{2n}{b}\right).
\end{equation}
In order to establish the desired bound on the bandwidth, we thus need to show
that $|S|+|P| \le n/\beta$. We first estimate the size of $S$. By
Claim~\ref{cl:sep-band} at most $x_S\le b-1$ separators have been extracted in
total, which implies
\begin{equation}\label{eq:sep-band:sep}
  |S|\le x_S\cdot\s(G)\le(b-1)\s(G).
\end{equation}
Furthermore all vertices $v\in P$ satisfy $\dist_G(v,S)\le \lfloor b/2\rfloor-1$
by definition. As~$G$ has maximum degree $\Delta$ there are at most
$|S|(\Delta^{\lfloor b/2\rfloor}-1)/(\Delta-1)$ vertices $v\in V \setminus S$
with this property and hence
\begin{equation*}\begin{split}  
 |S| + |P| 
 &\le|S| \left(1+\frac{\Delta^{\lfloor b/2\rfloor}-1}{\Delta-1}\right)
  \le |S| \frac{\Delta^{\beta/2}}{\Delta-3/2}\\
   &\le \frac{(b-1) \s(G)}{(\Delta-3/2)} \sqrt{\frac{n}{\s(G)}} 
   = \frac{(b-1) n}{(\Delta-3/2)} \sqrt{\frac{\s(G)}{n}}
\end{split}\end{equation*} 
where the second inequality holds for any $\Delta\ge3$ and $b\ge6$ and the third
inequality follows from~\eqref{eq:sep-band:bt} and~\eqref{eq:sep-band:sep}.
It is easy to verify that for any $\Delta\ge3$ and $x\ge\Delta^6$ we have
$(\Delta-3/2)\sqrt{x}\ge\tfrac98\log^2_{\Delta} x$. This together
with~\eqref{eq:sep-band:bt} gives $(\Delta-3/2)\sqrt{n/\s(G)}\ge\tfrac98\beta^2$
and hence we get
\begin{align}\label{eq:sep-band:SP}
  |S| + |P| \le \frac{8(b-1) n}{9\beta^2}\,.
\end{align} 
As $6\le b = \lfloor\beta\rfloor$ it is not difficult to check that
\[
	\frac{8(b-1)}{9\beta^2} + \frac 2b \le\frac{3}{\beta} \,.
\]
Together with~\eqref{eq:sep-band:bdw} and~\eqref{eq:sep-band:SP} this gives our
bound.

It remains to prove Claim~\ref{cl:sep-band}. Notice that the process of
repeatedly separating $G$ and its subgraphs can be seen as a binary tree~$T$
on vertex set~$W$ whose internal nodes represent the extraction of a separator
$S_{i,k}$ for some $i$ (and thus the separation of a subgraph of $G$ into two sets
$V_{i,j}$ and $V_{i,j'}$) and whose leaves represent the sets $V_{i,j}$ that
are of size at most $2n/b$. Clearly the number of leaves of $T$ is $b_{i^*}$
and the number of internal nodes $x_S$. As $T$ is a binary tree we conclude
$x_S=b_{i^*}-1$ and thus it suffices to show that $T$ has at most $b$ leaves in
order to establish the claim. To this end we would like to apply
Proposition~\ref{prop:tree}. Label an internal node of $T$ that
represents a separator $S_{i,k}$ with $|S_{i,k}|/n$, a leaf representing $V_{i,j}$ with
$|V_{i,j}|/n$ and denote the resulting labelling by $\ell$. Clearly we have
$\sum_{w\in W}\ell(w)=1$. Moreover we claim that 
\begin{equation}\label{eq:sep-band:ell}
  \ell(w)+\sum_{u\in L(w)}\ell(w)\ge|L(w)|/b
  \qquad\text{for all $w\in W$}
\end{equation}  
where $L(w)$ denotes the set of leaves that are children of $w$. Indeed, let
$w\in W$, notice that $|L(w)|\le 2$ as $T$ is a binary tree, and let $u$ and
$u'$ be the two children of $w$. If $|L(w)|=0$ we are done. If $|L(w)|>0$ then
$w$ represents a $(2/3,s(G))$-separator $S(w):=S_{i-1,k}$ that separated a graph
$G[V(w)]$ with $V(w):=V_{i-1,j}\ge 2n/b$ into two sets $U(w):=V_{i,j'}$ and
$U'(w):=V_{i,j'+1}$ such that
$|U(w)|+|U'(w)|+|S(w)|=|V(w)|$. In the case that $|L(w)|=2$ this
implies
\begin{equation*}
  \ell(w)+\ell(u)+\ell(u')=\frac{|S(w)|+|U(w)|+|U'(w)|}{n}
  =\frac{|V(w)|}{n}
  \ge 2/b
\end{equation*}
and thus we get~\eqref{eq:sep-band:ell}. If $|L(w)|=1$ on the other hand then,
without loss of generality, $u$ is a leaf of $T$ and $|U'(w)|>2n/b$. Since
$S(w)$ is a $(2/3,s(G))$-separator however we know that $|V(w)|\ge\frac32|U'(w)|$ and
hence
\begin{equation*}\begin{split}
  \ell(w)+\ell(u)&=\frac{|S(w)|+|U(w)|}{n}
  =\frac{|S(w)|+|V(w)|-|U'(w)|-|S(w)|}{n} \\
  &\ge\frac{\frac32|U'(w)|-|U'(w)|}{n}
  \ge\frac{\frac12(2n/b)}{n}
\end{split}\end{equation*}
which also gives~\eqref{eq:sep-band:ell} in this case.
Therefore we can apply Proposition~\ref{prop:tree} and infer that $T$ has at
most $b$ leaves as claimed.
%
\end{proof}

\subsection{Boundedness}\label{sec:bound}

In this section we study the relation between boundedness, bandwidth and
treewidth. We first give a proof of Proposition~\ref{prop:band-bound}. 


\begin{proof}[of Proposition~\ref{prop:band-bound}]
We have to show that for every graph $G$ and every $\eps>0$ 
the inequality $b_{\eps}(G) \le 2 \bdw(G) / \eps$ holds. 
Suppose that $G$ has $n$ vertices and let $\sigma: V \to [n]$ be an
arbitrary labelling of $G$.  
Furthermore assume that $V' \subseteq V$
with $|V'| = b_{\eps}(G)$ induces an $\eps$-expander in $G$. 
Define $V^*\subset V'$ to be the first $b_{\eps}(G)/2=|V'|/2$ vertices of $V'$
with respect to the ordering $\sigma$. Since $V'$ induces an $\eps$-expander in
$G$ there must be at least $\eps b_{\eps}(G)/2$ vertices in $N^*:=N(V^*)\cap V'$.
Let $u$ be the vertex in $N^*$
with maximal $\sigma(u)$ and $v\in V^*\cap N(u)$. As $u\not\in V^*$ and
$\sigma(u')>\sigma(v')$ for all $u'\in N^*$ and $v'\in V^*$ by the choice of
$V^*$ we have $|\sigma(u)-\sigma(v)|\ge|N^*|\ge\eps b_{\eps}(G)/2$.
Since this is true for every labelling $\sigma$ we can deduce 
that $b_{\eps}(G) \le 2\bdw(G) / \eps$.
\end{proof}

The remainder of this section is devoted to the proof of
Theorem~\ref{thm:bound-trw}. We will use the following lemma which establishes
a relation between boundedness and certain separators.

\begin{lemma}[bounded$\to$separator] \label{lem:bound-sep}
Let $G$ be a graph on $n$ vertices and let $\eps > 0$. If $G$ is
$(n/2,\eps)$-bounded then $G$ has a $(2\eps n/3, 2/3)$-separator.
\end{lemma}

\begin{proof} 
Let $G = (V,E)$ with $|V|=n$ be $(n/2,\eps)$-bounded for $\eps>0$.
It follows that every subset $V' \subseteq V$ with $|V'|
\ge n/2$ induces a subgraph $G' \subseteq G$ with the following property:
there is $W \subseteq V'$ such that $|W| \le |V'|/2$ and $|N_{G'}(W)| \le \eps
|W|$. 
We use this fact to construct a $(2\eps n/3, 2/3)$-separator in the following
way:
\begin{enumerate}[leftmargin=*,itemsep=0mm,parsep=0mm,topsep=2mm]
  \item Define $V_1:=V$ and $i:=1$. 
  \item\label{it:Sep_Konstr_1} Let $G_i:=G[V_i]$.
  \item\label{it:Sep_Konstr_2} Find a subset $W_i \subseteq V_i$ with $|W_i| \le
  |V_i|/2$ and $|N_{G_i}(W_i)| \le \eps |W_i|$.
 \item Set $S_i:= N_{G_i}(W_i)$, $V_{i+1}:=V_i \setminus (W_i \cup S_i)$.
 \item If $|V_{i+1}| \ge\frac23n$ then set $i:=i+1$ and go to step (\ref{it:Sep_Konstr_1}).
 \item Set $i^*:=i$ and return
$$
A:=\bigcup_{i=1}^{i^*} W_i, \quad B:= V_{i^*+1}, \quad S:=\bigcup_{i=1}^{i^*} S_i .
$$
\end{enumerate}
This construction obviously returns a partition $V = A \dcup B \dcup S$
with $|B| <\frac23n$. 
Moreover,
$|V_{i^*}| \ge \frac23 n$ and $|W_{i^*}| \le |V_{i^*}|/2$ and hence
\begin{equation*}\begin{split}
|A| = n - |B| - |S| = n - |V_{i^*+1}| - |S| =  \\
n - ( |V_{i^*}| - |W_{i^*}| - |S_{i^*}|) - |S| \le
n - \frac{|V_{i^*}|}{2} \le \frac23 n.
\end{split}\end{equation*}
The upper bound on $|S|$ follows easily since
$$
|S| = \sum_{i=1}^{i^*} |N_{G_i}(W_i)| \le \sum_{i=1}^{i^*} \eps|W_i| 
= \eps |A| \le \frac23 \eps n. 
$$
It remains to show that $S$ separates $G$. This is indeed the case as
$N_G(A) \subseteq S$ by construction and thus $E(A, B) = \emptyset$.
\end{proof}
Now we can prove Theorem~\ref{thm:bound-trw}. As remarked earlier, 
Grohe and Marx~\cite{GroMar} independently gave a proof of an equivalent result 
which employs similar ideas but does not use separators explicitly.

\begin{proof}[of Theorem~\ref{thm:bound-trw}]
Let $G=(V,E)$ be a graph on $n$ vertices, $\eps > 0$, and let $b \ge
b_{\eps}(G)$. 
It follows immediately from the definition of boundedness that every subgraph $G'
\subseteq G$ with $G'=(V',E')$ and $|V'| \ge 2b$ also has $b_{\eps}(G') \le b$. 

We now prove Theorem~\ref{thm:bound-trw} by induction on the size of $G$.
The relation
$\trw(G) \le 2 \eps n + 2b$ trivially holds if $n\le 2b$. So let $G$ have $n >
2b$ vertices and assume that the theorem holds for all graphs with less than $n$
vertices. Then $G$ is $(b, \eps)$-bounded and thus has a $(2\eps
n/3,2/3)$-separator $S$ by Lemma~\ref{lem:bound-sep}. Assume that $S$ separates
$G$ into the two subgraphs $G_1=(V_1,E_1)$ and $G_2=(V_2,E_2)$. Let
$(\mathcal{X}_1,T_1)$ and $(\mathcal{X}_2,T_2)$ be tree decompositions of $G_1$
and $G_2$, respectively, such that $\mathcal{X}_1\cap\mathcal{X}_2=\emptyset$. We
use them to construct a tree decomposition $(\mathcal{X},T)$ of $G$ as follows.
Let $\mathcal{X} = \{ X_i \cup S : X_i \in\mathcal{X}_1\} \cup   \{ X_i \cup S :
X_i \in\mathcal{X}_2\}$ and $T = (I_1 \cup I_2, F = F_1 \cup F_2 \cup \{e\})$
where $e$ is an arbitrary edge between the two trees. This is indeed a tree
decomposition of $G$: Every vertex $v \in V$ belongs to at least one $X_i \in
\mathcal{X}$ and for every edge $\{v,w\} \in E$ there exists $i \in I$ (where $I$
is the index set of $\mathcal{X}$) with $\{v,w\} \subseteq X_i$. This is trivial
for $\{v,w\} \subseteq V_i$ and follows from the definition of $\mathcal{X}$ for
$v \in S$ and $w \in V_i$. Since $S$ separates $G$ there are no edges $\{v,w\}$
with $v \in V_1$ and $w \in V_2$. For the same reason the third property of a
tree decomposition holds: if $j$ lies on the path from $i$ to $k$ in $T$, then
$X_i \cap X_k \subseteq X_j$ as the intersection is $S$ if $X_i, X_k$ are subsets
of $V_1$ and $V_2$ respectively.

We have seen that $(\mathcal{X},T)$ is a tree decomposition of $G$ and
can estimate its width as follows: $\trw(G) \le \max\{\trw(G_1), \trw(G_2)\} +
|S|$. With the induction hypothesis we get
\begin{align*}
\trw(G) &\le \max\{2\eps\cdot|V_1|+2b,\,\,2\eps\cdot |V_2|+2b\} + |S|\\
 &\le 2 \eps n + 2b.
\end{align*}
where the second inequality follows from $|V_i| \le (2/3)n$
and $|S|\le (2\eps n)/3$. 
\end{proof}

\section{Applications}
\label{sec:appl}

For many interesting bounded degree graph classes (non-trivial) upper bounds
on the bandwidth are not at hand. A wealth of results however has been obtained
about the existence of sublinear separators. This illustrates the importance
of Theorem~\ref{thm:main}. In this section we will give examples of such
separator theorems and provide applications of them in
conjunction with Theorem~\ref{thm:main}.

\subsection{Separator theorems}\label{subsec:separators}

A classical result in the theory of planar graphs concerns the existence of
separators of size~$2\sqrt{2n}$ in any planar graph on~$n$ vertices proved by
Lipton and Tarjan~\cite{LipTar} in 1977. Clearly, together with
Theorem~\ref{thm:sep-band} this result implies the following theorem.

\begin{corollary}
  \label{cor:planar}
  Let $G$ be a planar graph on $n$ vertices with maximum degree at
  most~$\Delta\ge 2$. Then the bandwidth of $G$ is bounded from above by
  \[
   \bdw(G) \le \frac{15n}{\log_\Delta(n)}.
  \]
\end{corollary}
%
It is easy to see that the bound in Corollary~\ref{cor:planar} is sharp up to
the multiplicative constant -- since the bandwidth of any graph $G$ is bounded from below by
$(n-1)/\mbox{diam}(G)$, it suffices to consider for example the complete binary tree on $n$ vertices.
Corollary~\ref{cor:planar} is used in~\cite{CDMW} to infer a result about the
geometric realisability of planar graphs $G=(V,E)$ with $|V|=n$ and $\Delta(G)\le\Delta$.

This motivates why we want to consider some generalisations of the planar
separator theorem in the following. The first such result is due to Gilbert,
Hutchinson, and Tarjan~\cite{GilHutTar} and deals with graphs of arbitrary
genus.
 \footnote{Again, the separator theorems we refer to bound the size of a
   separator in $G$. Since the class of graphs with genus less than $g$
   (or, respectively, of $H$-minor free graphs) is closed under taking subgraphs however, this
   theorem can also be applied to such subgraphs and thus the bound on~$\s(G)$
   follows.}

\begin{theorem}[\cite{GilHutTar}]
\label{thm:sep-genus}
  An $n$-vertex graph $G$ with genus $g \ge 0$ has separation
  number $\s(G) \le6\sqrt{gn}+2\sqrt{2n}$.
\end{theorem}

For fixed $g$ the class of all graphs with genus at most $g$ is closed under
taking minors. Here $H$ is a minor of $G$ if it can be obtained from a
subgraph of~$G$ by a sequence of edge deletions and contractions. A graph $G$
is called $H$-minor free if $H$ is no minor of $G$. The famous graph minor theorem
by Robertson and Seymour~\cite{RobSeyXX} states that any minor closed class
of graphs can be characterised by a finite set of forbidden minors (such as
$K_{3,3}$ and $K_5$ in the case of planar graphs). The next separator
theorem by Alon, Seymour, and Thomas~\cite{AloSeyTho} shows that already
forbidding one minor enforces a small separator.

\begin{theorem}[\cite{AloSeyTho}]
\label{lem:sep-minor}
  Let $H$ be an arbitrary graph. Then any $n$-vertex graph $G$ that is 
  $H$-minor free has separation number $\s(G) \le |H|^{3/2}\sqrt{n}$.
\end{theorem}

We can apply these theorems to draw the following conclusion concerning the
bandwidth of bounded-degree graphs with fixed genus or some fixed forbidden
minor from Theorem~\ref{thm:sep-band}.

\begin{corollary}\label{cor:sep}
  Let $g$ be a positive integer, $\Delta\ge2$ and $H$ be
  an $h$-vertex graph and $G$ an $n$-vertex graph with maximum degree 
  $\Delta(G)\le\Delta$.
  \begin{enumerate}[label=\abc,leftmargin=*,itemsep=0mm,parsep=0mm,topsep=2mm]
    \item If $G$ has genus $g$ then $\bdw(G)\le 15n/\log_{\Delta}(n/g)$.
    \item If $G$ is $H$-minor free then $\bdw(G)\le 12n/\log_{\Delta}(n/h^3)$.
  \end{enumerate}
\end{corollary}

\subsection{Embedding problems and universality}\label{subsec:universal}

A graph $H$ that contains copies of all graphs $G\in\cG$ for some class of
graphs $\cG$ is called \emph{universal for~$\cG$}.
The construction of sparse universal graphs for certain families $\cG$ has
applications in VLSI circuit design and was extensively studied (see,
e.g., \cite{AloCap} and the references therein).
In contrast to these results our focus is not on
minimising the number of edges of $H$,
but instead we are interested in giving a relatively 
\emph{simple} criterion for
universality for $\cG$ that is satisfied by \emph{many} graphs~$H$ of the
same order as the largest graph in $\cG$.

The setting with which we are concerned here are embedding results that
guarantee that a bounded-degree graph $G$
can be embedded into a graph $H$ with sufficiently high minimum
degree, even when $G$ and $H$ have the same number of vertices. 
Dirac's theorem~\cite{Dir} concerning the existence of Hamiltonian cycles in 
graphs of minimum degree $n/2$ is a classical example for theorems of this type.
It was followed by results of Corr\'adi and
Hajnal~\cite{CorHaj}, Hajnal and Szemer\'edi~\cite{HajSze} about embedding
$K_r$-factors, and more recently by a series of theorems due to Koml\'os, Sark\"ozy, and
Szemer\'edi and others
which deal with powers of Hamiltonian cycles, trees, and $H$-factors (see,
e.g., the survey~\cite{KOSurvey}). Along the lines of these results the
following unifying conjecture was made by Bollob\'as and Koml\'os~\cite{Kom99}
and recently proved by B\"ottcher, Schacht, and Taraz~\cite{BST09}.

\begin{theorem}[\cite{BST09}]
  \label{thm:bolkom}
  For all $r,\Delta\in\mathbb{N}$ and $\gamma>0$, there exist constants $\beta>0$ and
  $n_0\in\mathbb{N}$ such that for every $n\geq n_0$ the following holds.
  If~$G$ is an $r$-chromatic graph on~$n$ vertices with $\Delta(G) \leq \Delta$ and
  bandwidth at most $\beta n$ and if~$H$ is a graph on~$n$ vertices with minimum degree
  $\delta(H) \geq (\frac{r-1}{r}+\gamma)n$, then $G$ can be embedded into $H$.
\end{theorem}

The proof of Theorem~\ref{thm:bolkom} heavily uses the bandwidth constraint
insofar as it constructs the required embedding sequentially, following the
ordering given by the vertex labels of $G$. Here it is of course beneficial that
the neighbourhood of every vertex $v$ in $G$ is confined to the $\beta n$
vertices which immediately precede or follow $v$.

Also, it is not difficult to see that the statement in Theorem~\ref{thm:bolkom}
becomes false without the constraint on the bandwidth: Consider $r=2$, let $G$ be
a random bipartite graph with bounded maximum degree and let $H$ be the graph
formed by two cliques of size $(1/2+\gamma)n$ each, which share exactly $2\gamma
n$ vertices. Then $H$ cannot contain a copy of $G$, since in $G$ every vertex set
of size $(1/2-\gamma)n$ has more than $2\gamma n$ neighbours. The reason for this
obstruction is again that $G$ has good expansion properties.

On the other hand, Theorem~\ref{thm:main} states that in bounded degree graphs,
the existence of a big expanding subgraph is in fact the only obstacle which can
prevent sublinear bandwidth and thus the only possible obstruction for a
universality result as in Theorem~\ref{thm:bolkom}. More precisely we immediately
get the following corollary from Theorem~\ref{thm:main}.

\begin{corollary}
  If the class $\cC$ meets one (and thus all) of the conditions in
  Theorem~\ref{thm:main}, then the following is also true.
  For every $\gamma>0$ and $r\in\N$ there exists $n_0$ such that for all $n\ge
  n_0$ and for every graph $G\in\cC_n$ with chromatic number $r$ and for every 
  graph $H$ on $n$ vertices with minimum degree at least
  $(\frac{r-1}{r}+\gamma)n$, the graph $H$ contains a copy of $G$. 
\end{corollary}

By Corollary~\ref{cor:planar} we infer as a special case that all sufficiently
large graphs with minimum degree $(\frac34+\gamma)n$ are universal for the class
of bounded-degree planar graphs. Universal graphs for bounded degree planar
graphs have also been studied in~\cite{BCLR,Cap}.

\begin{corollary}\label{cor:planar-univers}
  For all $\Delta\in\mathbb{N}$ and $\gamma>0$, there exists $n_0\in\mathbb{N}$
  such that for every $n\geq n_0$ the following holds:
  \begin{enumerate}[label=\abc,leftmargin=*,itemsep=0mm,parsep=0mm,topsep=2mm]
  \item
    Every $3$-chromatic planar graph on $n$ vertices with maximum degree at most
    $\Delta$ can be embedded into every graph on $n$ vertices with minimum degree at
    least $(\frac23+\gamma)n$.
  \item
    Every planar graph on $n$ vertices with maximum degree at most $\Delta$ can be
    embedded into every graph on $n$ vertices with minimum degree at least
    $(\frac34+\gamma)n$.
  \end{enumerate}
\end{corollary}

This extends a result by K{\"u}hn, Osthus, and Taraz~\cite{KueOstTar}, who
proved that for every graph $H$ with minimum degree at least $(\frac23+\gamma)n$
\emph{there exists a particular} spanning triangulation $G$ that can be
embedded into $H$. Using Corollary~\ref{cor:sep} it is moreover possible to
formulate corresponding generalisations for graphs of fixed genus and for
$H$-minor free graphs for any fixed~$H$.

\section{Acknowledgement}

The first author would like to thank David Wood for fruitful discussions
in an early stage of this project.
In addition we thank two anonymous referees for their helpful suggestions. 


\bibliographystyle{amsplain} 
\bibliography{bandwidth}

\end{document}